\documentclass[11pt]{amsart}
\usepackage {amsmath, amssymb, amscd, mathrsfs, amsthm, stmaryrd,  bbm, diagbox, enumerate, slashed, graphicx, color, subfig, transparent, comment, float}
\usepackage[dvipsnames]{xcolor}
\usepackage{url, hyperref}
\usepackage{footmisc}
\usepackage{tikz-cd}
\usepackage[margin=1.2in]{geometry}

\newtheorem{theorem}{Theorem}[section]
\newtheorem{proposition}[theorem]{Proposition}
\newtheorem{corollary}[theorem]{Corollary}
\newtheorem{lemma}[theorem]{Lemma}

\newtheorem{definition}[theorem]{Definition}
\theoremstyle{remark}
\newtheorem{example}[theorem]{Example}
\theoremstyle{definition}

\newtheorem {remark}[theorem]{Remark}

\newcommand{\CPbone}{{\overline{\mathbb{CP}^2}}}
\newcommand{\CPar}[1]{{\#^{#1}{\mathbb{CP}^2}}}
\newcommand{\CPbar}[1]{{\#^{#1}\overline{\mathbb{CP}^2}}}
\newcommand{\punc}{\setminus{\rm int}(B^4)}
\newcommand{\pint}[1]{\setminus{\rm int}(#1)}
\newcommand{\Sq}{{\rm Sq}}
\newcommand{\odd}{\rm odd}
\newcommand{\even}{\rm even}

\begin{document}
\title{slicing degree of knots}
\author{Qianhe Qin}
\address {Department of Mathematics, Stanford University\\
Stanford, California, 94305, United States of America}
\email {\href{mailto:qqhe@stanford.edu}{qqhe@stanford.edu}}

\begin{abstract}
The slicing degree of a knot $K$ is defined as the smallest integer $k$ such that $K$ is $k$-slice in {$\CPbar{n}$} for some $n$. In this paper, we establish bounds for the slicing degrees of knots using Rasmussen's $s$-invariant, knot Floer homology and singular instanton homology. We compute the slicing degree for many small knots (with crossing numbers up to $9$) and for some families of torus knots.
\end{abstract}
\maketitle

\section{Introduction}

RBG links, developed in \cite{MP21} (see also \cite{akbulut_1977},\cite{akbulut_1993}), are a tool for generating knot pairs with the same $0$-surgery, with the goal of finding an exotic {\small $ \CPar{n}$} for some $n$. Suppose there exists a knot pair $(K,K')$ associated to some RBG link, such that $K$ is $H$-slice in {\small $ \CPar{n}$} and $K'$ is not $H$-slice in {\small $ \CPar{n}$}. Then, we can build a $4$-manifold $X'$ by cutting-and-pasting the zero traces of the knot pair $(K,K')$ using their $0$-surgery homeomorphism. It is not hard to check that $X'$ is homeomorphic to {\small $ \CPar{n}$} by the classification of simply-connected $4$-manifolds \cite{Freedman} and $X'$ is not diffeomorphic to {\small $ \CPar{n}$} by the trace embedding lemma \cite{19trace}.  

The RBG technique is generalized to $k$-RBG links in \cite{Qin}, and this can be used to generate knot pairs that share the same $k$-surgeries. A knot $K$ is called $k$-slice if it bounds a smoothly embedded disk with self-intersection number $-k$ in {\small $ \CPbar{n}$} for some $n$. We can construct an exotic {\small $ \CPbar{n}$} if there exist knot pairs associated to a special kind of $k$-RBG link, such that $K$ is $k$-slice in {\small $ \CPbar{n}$} and $K'$ is not $k$-slice in {\small $ \CPbar{n}$}. For instance, if the knot $K_B(3)$ in Figure \ref{KBmKG} is $3$-slice, then an exotic {\small $ \CPbar{n}$} exists for some $n$ (see \cite[Theorem 1.5]{Qin}). Thus, we need more information on either obstructing or establishing the $k$-sliceness of knots. 

In this paper, we will provide methods to compute a knot concordance invariant of knots called {\em slicing degree}, which was originally defined by Nakamura in \cite{Kai}. (In \cite{Kai}, this was called the {\em positive projective slice framing}, denoted $\mathbb{PF}_+(K)$.)
\begin{definition}[Definition 2.9 \cite{Kai}]
Let $K$ be a knot in $S^3$. The {\em (positive) slicing degree} $sd_+(K)$ of $K$ is 
$$
min\left\{k\in\mathbb{Z}_{\ge 0}\left\vert\; \text{$K$ is $k$-slice in {\small $ \CPbar{n}$} for some $n$}\right\}\right. .
$$
\end{definition}
 
One can obtain upper bounds on the slicing degree of a knot $K$ from an unknotting sequence of negative generalized crossing changes (i.e. adding negative twists on parallel stands) on a knot diagram of $K$ (see \cite[Lemma 2.8]{Kai}). Moreover, since the slicing degree is a concordance invariant, we can strengthen the upper bound by combining unknotting operations and knot concordances. 
\begin{proposition} \label{prop:cs+}
Let $c_s^+(K)$ be the positive clasp number of $K$. Then,
$$sd_+(K)\le 4 \cdot c_s^+(K).$$
\end{proposition}
Since the positive clasp number is bounded above by the unknotting number, the slicing degree is a finite number for each knot.

On the other hand, we can use knot invariants from Khovanov homology, knot Floer homology and singular instanton homology to bound the slicing degree from below. 

First, the adjunction inequality for the $s$-invariants \cite[Corollary 1.5]{Ren} gives an obstruction for a knot $K$ to bound a disk in {\small $ \CPbar{n}$}.
\begin{theorem}\label{prop:sinv}
If $K$ bounds a disk $D$ in {$ \CPbar{n}\punc$} such that $[D]=(a_1,\cdots, a_n)$ in $H^2(\CPbar{n}; \mathbb{Z})$, then
$$
s_p(K) \le \sum_{i=1}^na_i^2 - \sum_{i=1}^n|a_i|,
$$
where $s_p(K)$ denotes the $s$-invariant over characteristic $p$. 
\end{theorem}
From this, we can compute the slicing degree for the torus family $T_{m,m+1}$.
\begin{proposition}\label{prop:tmm+1}
Let $m>1$ be an integer. Then, $$sd_+(T_{m,m+1})=m^2.$$
\end{proposition}

Let $W$ be a negative-definite smooth $4$-manifold with $b_1(W)=0$ and $b_2(W)=n$. Denote $W\punc$ by $W^{\circ}$. If a knot $K$ bounds a disk $D$ in $W^{\circ}$, then the disk-exterior of $D$ in $W^{\circ}$ is a negative-definite 4-manifold bounded by a rational homology sphere. We get the following result by combining Ozsv{\'{a}}th and Szab{\'{o}}'s theorem \cite[Theorem 9.6]{OSdinv} with Ni and Wu's formula in \cite[Proposition 1.6]{NW}, which relates the $d$-invariant of surgeries on knot $K$ and the knot invariants $\{V_s(K)\}_{s\ge 0}$.

\begin{theorem}\label{lem:V_s}
Suppose $K$ bounds a disk $D$ in $W^{\circ}$ such that $[D]=(a_1,\cdots, a_n)$. Then, for $(\lambda_1,\cdots,\lambda_n)$ such that $\lambda_i$'s are odd and $0\le \sum_{i=1}^n \lambda_i a_i \le k$, we have that
$$\sum_{i=1}^n \lambda_i^2-n\ge 8V_j(K),$$
where $j=\frac{1}{2}(k-\sum \lambda_ia_i)$.
\end{theorem}
In particular, Theorem \ref{lem:V_s} applies to {\small $ \CPbar{n}$}. Here is a corollary.
\begin{theorem}\label{thm:t22m+1}
If $K$ is a Floer thin knot, then $sd_+(K)$ is greater or equal to $4\tau(K)$. In particular, equality is achieved by $T_{2, 2m+1}$, i.e. $$sd_+(T_{2, 2m+1}) = 4m.$$
\end{theorem}
Theorem \ref{lem:V_s} is also capable of narrowing down the slicing degree for the majority of knots with small crossing numbers (see Section \ref{small}). 

The concordance invariant $\Gamma_K(s)$ constructed by Daemi and Scaduto from the Chern-Simons filtration on the singular instanton Floer complex in \cite{daemi19} can be used to obstruct a knot from bounding a disk in certain homology classes.
\begin{proposition}
Let $K$ be a knot such that $\sigma(K) \le 0$. Suppose $K$ bounds a disk $D$ in $W^{\circ}$ such that $[D] = (\,\underbrace{2,\cdots, 2}_{p}\: , \underbrace{1,\cdots, {1}}_{q}\:, 0 \cdots, 0)$ in $H^2(W;\mathbb{Z})$ with $p,q\ge 0$ and $n \ge p+q$. Then, $$\Gamma_K\left(-\frac{1}{2}\sigma(K) \right) \le \frac{p}{2} + \frac{q}{8}.$$ 
\end{proposition}
Daemi and Scaduto computed the $\Gamma_K(s)$ when $K$ is a double twist knot in Corollary 3.24 \cite{daemi19}. This gives us finer lower bounds for the knot $7_4$ and $9_5$ (see Example \ref{ex: 7495}).

We can learn more about the $k$-sliceness of a knot $K$ with the help of the $k$-friends of $K$: knots that share the same $k$-surgery with $K$.
\begin{theorem}
Let $L=\{(R,r),(B,0),(G,0)\}$ be a $k$-special RBG link such that $R$ is $r$-slice in some $\#^{n}{\mathbb{CP}^2}$, and let $(K,K')$ be the knot pair associated to $L$. If $s(K') > k-\sqrt{k}$, then $sd_+(K)> k$.
\end{theorem}
From this, we construct examples of knots whose slicing degree is $1,2$ or $3$.
\begin{proposition}
Define $K_B(m)$ by the knot diagram to the left of Figure \ref{KBmKG}. Then,  
$$sd_+(K_B(m))=m+1,$$
for $m\in\{-1,0,1,2\}$.
\end{proposition}
Finally, we combine all the methods and compute the slicing degrees of almost all small knots (with crossing number less than 10); see Table \ref{08} and Table \ref{09}.
\medskip

{\bf Organization of the paper.} In Section \ref{basic}, we examine formal properties of the slicing degree, and provide upper bounds on slicing degrees by positive clasp numbers. In Section \ref{lowb}, we discuss methods to bound the slicing degree from below and compute the slicing degree for certain torus knots. In Section \ref{friends}, we investigate special relationships between knots and demonstrate the advantages of the RBG technique in obstructing $k$-sliceness. In Section \ref{small}, we compute the slicing degree for small knots.

{\bf Conventions.} All manifolds are smooth. All disks are smoothly, properly embedded. All manifolds are oriented and all maps are orientation preserving.

{\bf Acknowledgements.} The author is grateful to Ciprian Manolescu for suggesting this problem and offering continuous guidance. Furthermore, the author thanks Aliakbar Daemi, Adam Levine, Lisa Piccirillo, Qiuyu Ren and Hongjian Yang for helpful conversations.


\section{Basic properties}\label{basic}
In this section, we review the definition and demonstrate several basic properties of the slicing degree.
\begin{definition}[Definition 2.9 \cite{Kai}]
Let $K$ be a knot in $S^3$. The {\em (positive) slicing degree} $sd_+(K)$ of $K$ is 
$$
min\left\{k\in\mathbb{Z}_{\ge 0}\left\vert\; \text{$K$ is $k$-slice in {\small $ \CPbar{n}$} for some $n$}\right\}\right..
$$
\end{definition}
\begin{remark}
One could also define the negative slicing degree $sd_-(K)$ to be the maximum among all integers $k$ such that $K$ is $k$-slice in {$\CPar{n}$} for some $n$. The terminology positive and negative comes from the fact that $sd_+\ge 0$, whereas $sd_-\le 0$. Since $sd_-(K)=-sd_+(m(K))$, it suffices to consider the positive slicing degree for each knot together with its mirror. 
\end{remark}
One of the properties of the slicing degree is that the slicing degree of a knot $K$ vanishes if and only if $K$ is $H$-slice in some {$\CPbar{n}$} \cite[Lemma 2.10]{Kai}. Moreover, the slicing degree is a concordance invariant.
\begin{proposition}\label{concor}
If $K_0$ is concordant to $K_1$, then $sd_+(K_0)= sd_+(K_1)$.
\end{proposition}
\begin{proof}
Let $C$ be a concordance from $K_0$ to $K_1$ in $S^3 \times I$. The knot $K_1$ bounds a disk $D$ with self-intersection $-sd_+(K_1)$ in some {\small $\CPbar{n}\punc$}. Concatenate $S^3 \times I$ with {\small $\CPbar{n}\punc$}. Consider the disk $C\cup D$ with boundary $K_0$. Since the self-intersection number of $C\cup D$ is $-sd_+(K_1)$, we have that $sd_+(K_0) \le sd_+(K_1)$. Similarly, exchange $K_0$ and $K_1$, and obtain that $sd_+(K_1) \le sd_+(K_0)$. The result follows.
\end{proof}
\begin{proposition}
Given two knots $K_1$ and $K_2$,
$$sd_+(K_1\# K_2)\le sd_+(K_1) + sd_+(K_2).$$
\end{proposition}
\begin{proof}
There exists a disk $D_i$ in {\small $\CPbar{n_i}\punc$} with boundary $K_i$ such that the self-intersection number of $D_i$ in equals to $-sd_+(K_i)$. Take the boundary sum of {\small $\CPbar{n_1}\punc$} and  {\small $\CPbar{n_1}\punc$} away from the knots. Then we know that the disjoint union $K_1\sqcup K_2 $ bounds two disjoint disks in {\small $\CPbar{n_1+n_2}\punc$}. Connect $K_1$ and $K_2$ with a band to merge $D_1$ and $D_2$ into one disk $D$ whose boundary is the connected sum $K_1\# K_2$. Since $D$ has self-intersection number {\small $-(sd_+(K_1) + sd_+(K_2))$}, the inequality follows.
\end{proof}
However, the slicing degree is not additive. Consider a knot $K$ whose slicing degree is nonzero (see Table \ref{08}). Since $K\# m(K)$ is slice, we have that $sd_+(K\# m(K))=0$, which is strictly less than the sum $sd_+(K) + sd_+(m(K))$.

The slicing degree is finite with the following upper bound. Recall that the {\em slicing number} $u_s(K)$ of a knot $K$ is defined to be is the minimal number of crossing changes to turn $K$ into a slice knot, see \cite{owens}.
\begin{lemma}\label{lem:sn} 
Let $u_s(K)$ be the slicing number of $K$. Then,
$$sd_+(K)\le 4 \cdot u_s(K).$$
\end{lemma}
\begin{proof}
Let $K_s$ be the slice knot obtained by applying $u_s(K)$ many crossing changes on $K$. Since a crossing change can be realized by adding a full twist along two parallel strands, the knot $K$ can be obtained from $K_s$ by performing positive full twists on $u_s(K)$ disks, each of which intersects $K_s$ geometrically twice. Let $u_{s}^+(K)$ be the number of disks which intersect with $K$ algebraically twice. Since $K_s$ is slice in $S^4$, the knot $K$ is {\small $4u_{s}^+(K)$}-slice in {\small $\CPbar{u_s(K)}$} by \cite[Lemma 2.8]{Kai}. Thus, $sd_+(K)\le 4u_{s}^+(K)\le 4u_s(K)$.
\end{proof}
The {\em $4$-dimensional clasp number} (or $4$-ball crossing number) $c_s(K)$ is defined to be the minimal number of double points of an immersed disk in the $4$-ball whose boundary is $K$. By \cite[Proposition 2.1]{owens}, Owens and Strle showed that $$c_s(K)=\min_{K_c \text{ is concordant to } K} u_s(K_c).$$ Together with Lemma \ref{lem:sn} and Proposition \ref{concor}, we have the following upper bound.

\begin{proposition} Let $c_s(K)$ be the 4-dimensional clasp number of $K$. Then,
$$sd_+(K)\le 4 \cdot c_s(K).$$
\end{proposition}
\begin{proof}
Let $K_c$ be a knot concordant to $K$. By Lemma \ref{lem:sn}, $sd_+(K_c)\le 4u_s(K_c)$. By Proposition \ref{concor}, $sd_+(K)=sd_+(K_c)\le 4u_s(K_c)$. Since this holds for any $K_c$ that is concordant to $K$, we conclude that $sd_+(K)\le 4 \cdot c_s(K).$
\end{proof}
We can get tighter bounds if we keep track of the signs of the crossing changes as in Example \ref{+casp} below. The {\em positive clasp number} $c_s^+(K)$ is the minimal number of positive double points realized by a normally immersed disk in the $4$-ball with boundary $K$. 
{
\renewcommand{\thetheorem}{\ref{prop:cs+}}
\begin{proposition} 
Let $c_s^+(K)$ be the positive clasp number of $K$. Then,
$$sd_+(K)\le 4 \cdot c_s^+(K).$$
\end{proposition}
\addtocounter{theorem}{-1}
}
\begin{proof}

Consider an immersed disk $D$ with boundary $K$ which has $c_s^+(K)$-many positive double points. Choose a neighborhood for each double point of D, which intersects the boundary of the neighborhoods along Hopf links. Since a Hopf link bounds two disjoint disks in a punctured {\small $\CPbone$}, we resolve the singularities by gluing back a  {\small $\CPbone\punc$} and capping off the Hopf links by two disjoint disks. Notice that the disks in  {\small $\CPbone\punc$} with boundary a positive Hopf link is of homology class $(2)$, while the disks bounded by a negative Hopf link is null-homologous. Thus, the resolved embedded disk has self-intersection number $-4 \cdot c_s^+(K)$ in some {\small $\CPbar{n}$}.
\end{proof}

\begin{proposition}
If $sd_+(K)=4u_s(K)$, then $sd_+(m(K))=0$. Similarly, if $sd_+(K)=4u_c(K)$, then $sd_+(m(K))=0$.
\end{proposition}

\begin{proof}
Pick the diagram $D_K$ of $K$, which represents a slice knot after $u_s(K)$-many crossing changes. Since $sd_+(K)=4u_s(K)$, we must have that all those crossing changes are positive in $D_K$. Thus, take the mirror diagram $m(D_K)$, which represents the mirror of $K$. Change the same crossings as above. Since we only changed negative crossings to positive crossings, we have thus constructed an $H$-slice disk for $m(K)$.
\end{proof}

However, there exist knots $K$ such that both $sd_+(K)$ and $sd_+(m(K))$ are nonzero, see Example \ref{s0s2} and Example \ref{v00} below.

\section{Lower bounds}\label{lowb}
In this section, we explore methods of obstructing a knot $K \subset S^3$ from bounding a disk with certain homology class in $\CPbar{n}\punc$, and give lower bounds for the slicing degree of $K$. 

\subsection{Rasmussen's $s$-invariants}
We use $s_p(K)$ to denote the Rasmussen's $s$-invariant of a knot $K$ over characteristic $p$. For $p=0$, we denote the Rasmussen's $s$-invariant over $\mathbb{Q}$ by either $s(K)$ or $s_0(K)$. Let $K$ be a knot that bounds a disk in $\#^n\overline{\mathbb{CP}^2}$ of class $(a_1,\cdots,a_n)$ in $H^2(\#^n\overline{\mathbb{CP}^2})$. Applying \cite[Corollary 1.5]{Ren} to the disk $D$, we obtain the following proposition. 
{
\renewcommand{\thetheorem}{\ref{prop:sinv}}
\begin{theorem}
If $K$ bounds a disk $D$ in {$ \CPbar{n}\punc$} such that $[D]=(a_1,\cdots, a_n)$ in $H^2(\CPbar{n}; \mathbb{Z})$, then
$$
s_p(K) \le \sum_{i=1}^na_i^2 - \sum_{i=1}^n|a_i|,
$$
where $s_p(K)$ denotes the $s$-invariant over characteristic $p$. 
\end{theorem}
\addtocounter{theorem}{-1}
}
Furthermore, one can obtain lower bounds for the slicing degree as follows.
\begin{proposition}\label{stau}
If $s_p(K)$ is positive for some characteristic $p$, then
$$
sd_+(K)\ge s_p(K) + \frac{1}{2} + \sqrt{s_p(K)+\frac{1}{4}}
$$
\end{proposition}
{
\renewcommand{\thetheorem}{\ref{prop:tmm+1}}
\begin{proposition}
Let $m$ be an integer such that $m>1$. Then, $sd_+(T_{m, m+1})=m^2$.
\end{proposition}
\addtocounter{theorem}{-1}
}
\begin{figure}[h]
{
   \fontsize{9pt}{11pt}\selectfont
   \def\svgwidth{1.6in}
   \begin{center}
\begingroup%
  \makeatletter%
  \providecommand\color[2][]{%
    \errmessage{(Inkscape) Color is used for the text in Inkscape, but the package 'color.sty' is not loaded}%
    \renewcommand\color[2][]{}%
  }%
  \providecommand\transparent[1]{%
    \errmessage{(Inkscape) Transparency is used (non-zero) for the text in Inkscape, but the package 'transparent.sty' is not loaded}%
    \renewcommand\transparent[1]{}%
  }%
  \providecommand\rotatebox[2]{#2}%
  \newcommand*\fsize{\dimexpr\f@size pt\relax}%
  \newcommand*\lineheight[1]{\fontsize{\fsize}{#1\fsize}\selectfont}%
  \ifx\svgwidth\undefined%
    \setlength{\unitlength}{186.25574067bp}%
    \ifx\svgscale\undefined%
      \relax%
    \else%
      \setlength{\unitlength}{\unitlength * \real{\svgscale}}%
    \fi%
  \else%
    \setlength{\unitlength}{\svgwidth}%
  \fi%
  \global\let\svgwidth\undefined%
  \global\let\svgscale\undefined%
  \makeatother%
  \begin{picture}(1,1.21900848)%
    \lineheight{1}%
    \setlength\tabcolsep{0pt}%
    \put(0,0){\includegraphics[width=\unitlength,page=1]{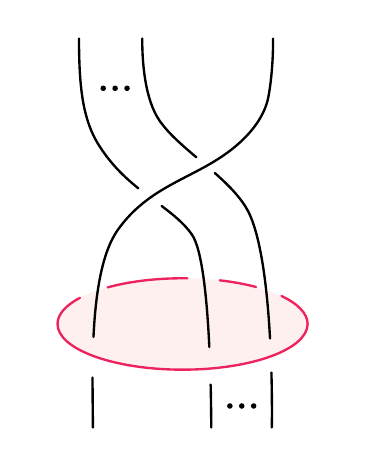}}%
    \put(0.86516296,0.39040113){\color[rgb]{0.92941176,0.1372549,0.37254902}\makebox(0,0)[lt]{\lineheight{1.25}\smash{\begin{tabular}[t]{l}$-1$\end{tabular}}}}%
    \put(0,0){\includegraphics[width=\unitlength,page=2]{torus.pdf}}%
  \end{picture}%
\endgroup%

   \end{center}
   \caption{}
   \label{torus}

}
\end{figure}
\begin{proof}
The $s$-invariant of the torus knot $T_{m,m+1}$ is $m(m-1)$. When $m>1$, by Proposition \ref{stau}, we have that $sd_+(T_{m, m+1})\ge m^2$. On the other hand, this lower bound can be achieved by a disk with self-intersection $-m^2$ in $\overline{\mathbb{CP}^2}$ as follows. 

Attach a $(-1)$-framed $2$-handle along $m$ vertical strands whose closure is an unknot (see Figure \ref{torus}). This gives us a Kirby diagram of $\overline{\mathbb{CP}^2}\pint{B^4 \cup B^4}$ with the unknot lying on the lower boundary. Blow down the $(-1)$-framed unknot and obtain the torus knot $T_{m,m+1}$ on the upper boundary.  Now, construct the disk by capping off the lower boundary of the annulus traced by the knot $T_{m, m+1}$. Since the disk intersect the core of the two handle algebraically $m$ times, the disk has self-intersection $-m^2$. 
\end{proof}

Cotton Seed first found knots whose $s_2$-invariant is different from the $s$-invariant, with knot $K14n19265$ being one of those knots. Note that the adjunction inequality \cite[Corollary 1.5]{Ren} for the $s$-invariant holds over any coefficient field. Thus, one can build knots $K$ such that both $K$ and the mirror $m(K)$ have arbitrarily large slicing degree.
\begin{example}\label{s0s2} 
Let $K$ be the $3$-twisted Whitehead double on the right-handed trefoil $W_3^+ (T_{2,3})$. From Table 4 \cite{lewark}, we have that $s_2(K) = 2$ and $s(K) = 0$. Consider $$J_m = \#^{2m} K \#^m T_{-2,3}.$$ Since $s_2(J_m)=2m$ and $s(J_m) = -2m$, we have that both $sd_+(J_m)$ and $sd_+(m(J_m))$ is greater than $2m$ by Proposition \ref{stau}.
\end{example}

From Theorem \ref{prop:sinv}, we can get sharper bound than that in Proposition \ref{stau} if we work with specific homology classes. 
\begin{proposition} \label{813}
 If $s_p(K)=4$ for some characteristic $p$, then $sd_+(K) \ge 8$. If $s_p(K)=8$ for some characteristic $p$, then $sd_+(K) \ge 13$.
\end{proposition}
\begin{proof}
Suppose $s_p(K)=4$. By Proposition \ref{stau}, we have that $sd_+(K) \ge 7$. One can write $7$ as the sum of squares in two ways: $7=1+\cdots+1=2^2+1+1+1$. However, none of them does not satisfy the inequality in Theorem \ref{prop:sinv}. Thus, $sd_+(K) \ge 8$.

Similarly, suppose $s_p(K)=8$. By Proposition \ref{stau}, we have that $sd_+(K) \ge 12$. Suppose $K$ is $k$-slice in some $\#^n\overline{\mathbb{CP}^2}$. For $k=12$, we have that $k=1+\cdots +1=2^2+1+\cdots +1=2^2+2^2+1+\cdots +1=2^2+2^2+2^2=3^2+1+1+1$. Theorem \ref{prop:sinv} obstructs the existence of slicing disks in such homology classes. Thus, $sd_+(K) \ge 13$.  
\end{proof}

Let $W$ be a negative-definite 4-manifold with $b_1(W)=0$ and $b_2(W)=n$. Suppose $K$ bounds a disk $D$ in $W\punc$ such that the homology class of $D$ is $(a_1,...,a_n)$ in $H^2(W;\mathbb{Z})$. Let $\beta$ be a knot invariant such that the adjunction inequality 
\begin{equation}\label{adjuntau}
2\beta(K) \le \sum_{i=1}^na_i^2 - \sum_{i=1}^n|a_i|,
\end{equation} 
holds for any choice of $K$, $n$ and $a_i$'s. Examples of $\beta$ are $s_p$ for any $p$. Others are $2\tau$ and $2\nu_+$, see Corollary \ref{nu+}. For small values of $\beta$, we tabulate the lower bounds of the slicing degree, along with the homology class (excluding zero terms) where the bounds are achieved.
\begin{center}

\begin{tabular}{||c | c | c ||} 
 \hline 
 $\beta(K)$ & $sd_+(K)\ge$ & $(a_1,\cdots,a_n)$ \\ [0.5ex] 
 \hline\hline
 2 & $ 4$ & $(2)$  \\ 
 \hline
 4 & $ 8$ & $(2,2)$ \\
 \hline
 6 & $ 9$ & $(3)$\\
 \hline
 8 & $ 13$ & $(2,3)$  \\
  \hline
 10 & $ 16$ & $(4)$\\
 \hline
 12 & $ 16$ & $(4)$\\
   \hline
   14 & $ 20$ & $(2,4)$ \\
 \hline
 16 & $ 24$ & $(2,2,4)$\\
 \hline
\end{tabular}
\end{center}

\subsection{Heegaard Floer invariants $\{V_s\}_{s\ge 0}$}\label{heegaard}
 Let $W$ be a negative-definite smooth 4-manifold with $b_1(W)=0$ and $b_2(W) = n$. Suppose there exists a disk $D$ such that $(D,K)$ is smoothly properly embedded in $(W^{\circ}, \partial W^{\circ})$, where $W^{\circ}=W\punc$. Denote the homology class of $D$ in $H^2(W)$ by $(a_1,\cdots, a_n)$ and let $k$ be $\sum_{i=1}^na_i^2$. Then the self-intersection number of $D$ is $[D]\cdot[D]= -\sum_{i=1}^na_i^2=-k$. 

The inequality in \cite[Theorem 9.6]{OSdinv} obstructs the disk $D$ from being in certain non-zero homology class (i.e, $k\neq 0$). Denote the disk exterior $W^{\circ}\setminus \nu(D)\cong W\setminus -X_k(K)$ by $E(D)$. Since $E(D)$ is negative-definite and bounds a rational homology ball $S^3_k(K)$, we apply Theorem 9.6 in \cite{OSdinv} to the pair $(E(D), S^3_k(K))$.

\begin{proposition}\label{thm:dinv}
Suppose $K$ bounds a disk $D$ in $W^{\circ}=W \punc$ such that $[D]=(a_1,\cdots,a_n) \neq 0$. Then, for any spin$^c$-structure $\mathfrak{s}=(\lambda_1,\cdots, \lambda_n)$ such that $\lambda_i$'s are odd numbers, 
\begin{equation}\label{dinv}
- \sum_{i=1}^n \lambda_i^2 + \frac{1}{k} (\sum_{i=1}^n \lambda_i a_i)^2 + n - 1 \le 4d(S_k^3(K), \mathfrak{t}),
\end{equation}
where $\mathfrak{t}$ is a restriction of the spin$^c$ structure $\mathfrak{s}$ on $W$ to $\partial E(D)$.
\end{proposition}

\begin{proof}
First, compute the rank of $H^2(E(D);\mathbb{Z})$. Take a tubular neighborhoods $\nu(D)$ of $D$ in $W^{\circ}$.  Shrink $\nu(D)$ uniformly in the normal direction and denote the smaller tubular neighborhood by $\nu^{\prime}(D)$. Let $W^{\circ}\pint{\nu^{\prime}(D)}$ be the disk exterior $E(D)$. The intersection $C=E(D) \cap {\nu}(D)$ is $S^1\times D^1\times D^2$. Consider the Mayer–Vietoris sequence given by $W^{\circ}=E(D)\cup {\nu}(D)$ as follows.
\begin{center}
\begin{tikzcd}[cramped, sep=scriptsize]
  0\cong H_2(C) \arrow[r] & H_2(E(D))\oplus H_2(D^2\times D^2) \arrow[r]\arrow[d,"\cong"] &H_2(W^{\circ})\arrow[r,"\partial"]\arrow[d,"\cong"] &H_1(C)\arrow[d,"\cong"]\arrow[r] &H_1(E(D))\arrow[r] &0\\
   & H_2(E(D))\arrow[r] &\mathbb{Z}^n \arrow[r,"\alpha"] & \mathbb{Z}
\end{tikzcd}
\end{center}
Here, the map $\alpha$ is given by taking inner product with $[D]=(a_1,\cdots ,a_n)$, i.e. $\alpha(x_1,\cdots,x_n)=\sum_i a_ix_i$. Since $[D]\neq 0$, we have that $rk(H_2(E(D); \mathbb{Z})=n-1$ and $H_1(E(D); \mathbb{Z})\cong \mathbb{Z}/l$, where $l$ is the greatest common divisor of $a_i$'s. It follows from Universal Coefficient Theorem that $rk(H^2(E(D); \mathbb{Z}))=n-1$.

Next, pick a spin$^c$ structure on $E(D)$. The spin$^c$ structures on a $E(D)$ correspond to characteristic elements in $H^2(E(D);\mathbb{Z})$. 
Consider the exact sequence:
\begin{center}
\begin{tikzcd}[cramped, sep=scriptsize]
H^2(W)\arrow[r] & H^2(E(D))\oplus H^2(X_k(K))\arrow[r, "\beta"] &H^2(S^3_k(K))\arrow[r] &H^3(W)\cong 0.
\end{tikzcd}
\end{center}
Since each spin$^c$ structure on $S^3_k(K)$ extends over the trace $X_n(K)$, we have that each spin$^c$ structure on $E(D)$ is a restriction from a spin$^c$ structure on $W$. Each characteristic element $\mathfrak{s}$ in $H^2(W)$ is of the form $(\lambda_1, \cdots, \lambda_n)$ with each $\lambda_i$ being an odd integer. Denote the restriction of the spin$^c$ structure $\mathfrak{s}=(\lambda_1, \cdots, \lambda_n)$ to $E(D)$ by $\mathfrak{s}|_{E(D)}$.

Since $S^3_k(K)$ is a rational homology sphere, $c_1\left(\mathfrak{s}|_{E(D)}\right)^2 + c_1\left(\mathfrak{s}|_{-X_k(K)}\right)^2=c_1(\mathfrak{s})^2$, which equals to $-\sum \lambda_i^2$. Let $\widehat{F}$ be the cocore of the 2-handle in $-X_k(K)$ capped with a Seifert surface $F$ of the knot $K$ in $S^3$. Under the isomorphism $H^2(-X_k(K);\mathbb{Z})\cong\mathbb{Z}[\widehat{F}]\cong \mathbb{Z}$, we have that $c_1\left(\mathfrak{s}|_{-X_k(K)}\right)=-\sum \lambda_i a_i$ and that the intersection form on $-X_k(K)$ is $(-\frac{1}{k})$.
Thus, one computes that $c_1\left(\mathfrak{s}|_{E(D)}\right)^2=-\sum \lambda_i^2+\frac{1}{k}(\sum \lambda_i a_i)^2$.

The result follows from applying \cite[Theorem 9.6]{OSdinv} to $E(D)$, i.e.
$$
c_1(\mathfrak{s}|_{E(D)})^2 + rk(H^2(E(D);\mathbb{Z}))\le 4d(S_k^3(K),\mathfrak{t}).
$$
\end{proof}

Recall that the invariants $\{V_s\}_{s\ge 0}$ form a knot $K$ is a non-increasing sequence of  non-negative integers, defined out of the full knot Floer complex of $K$. Denote the full complex $CFK^{\infty}(K)$ by $C$, and let $A_s^{+}=C\{\max\{i,j-s\}\ge 0\}$ and $B^{+}=C\{i\ge 0\}$. Consider the projection chain map $v_s^{+}: A_s^+ \rightarrow B^+$ as in \cite{OS08}, and let $\mathfrak{v}_s^+$ be the induced map on homology by $v_s^+$. Since $\mathfrak{v}_s^{+}$ is a graded isomorphisms at sufficiently high gradings and is $U$-equivariant, it is modeled on multiplication by $U^{V_s}$. 
Ni and Wu in \cite{NW} computed a formula for the $d$-invariants of $S^3_{p/q}(K)$ ($p,q>0$) in terms of the $d$-invariants of the lens space $L(p,q)$ and the $\{V_s\}_{s\ge 0}$ invariants of the knot $K$. 
Thus, the inequality \eqref{dinv} can be rewritten as follows.
\begin{proposition}\label{opti}
Suppose $K$ bounds a disk $D$ in $W^{\circ}$ such that $[D]=(a_1,\cdots,a_n) \neq 0$. Then, for any $(\lambda_1,\cdots, \lambda_n)$ such that $\lambda_i$'s are odd numbers,
\begin{equation}\label{family}
k(\sum_{i=1}^n \lambda_i^2)- (\sum_{i=1}^n \lambda_i a_i)^2+(2j-k)^2 - kn \ge 8k \cdot \max\{V_j(K),V_{k-j}(K)\},
\end{equation}
where $j$ is in $\mathbb{Z}/k\mathbb{Z}$ such that $\sum_{i=1}^n \lambda_i a_i \equiv k-2j \pmod{2k}$.
\end{proposition}
\begin{proof}
Let $\mathfrak{t}_j$ be the spin$^c$ structure on $S_k^3(K)$, which extends to a spin$^c$ structure $\mathfrak{s}$ such that $\langle c_1(\mathfrak{s}),[\widehat{F}]\rangle \equiv k+2j \pmod {2k}$. \cite[Theorem 2.1]{Bod} computes that 
$$
d(S_k^3(K), \mathfrak{t}_j) = -2\max\{V_j(K),V_{k-j}(K)\} + \frac{(k-2j)^2-k}{4k}
$$
Suppose that $\mathfrak{t}$ is a restriction of the spin$^c$ structure $\mathfrak{s}=(\lambda_1,\cdots,\lambda_n)$ as in Proposition \ref{thm:dinv}. Consider $\mathfrak{s}|_{-X_k(K)}$ on $-X_k(K)$, which restrict to $\mathfrak{t}$ on $-S^3_k(K)$. Let $\widehat{F}$ be the surface in $-X_k(K)$ obtained by capping off a Seifert surface $F$ of the knot $K$. Since $\langle c_1\left(\mathfrak{s}_{-X_k(K)}\right), [\widehat{F}]\rangle= \langle c_1(\mathfrak{s}), [\widehat{F}]\rangle=-\sum \lambda_i a_i$, we have that $\mathfrak{t}=\mathfrak{t}_j$, where $j\in\mathbb{Z}/k\mathbb{Z}$ satisfies that $k+2j \equiv-\sum \lambda_i a_i \pmod {2k}$ up to multiplication by $\pm 1$. Then, apply Proposition \ref{thm:dinv}.
\end{proof}

Given any homology class $[D]$ of the disk $D$, for each the spin$^c$ structure (i.e. $\lambda_i$'s) from the odd-integer lattice, there exists an inequality as in Proposition \ref{opti}. Thus, we can consider optimizing those inequalities among all the $\lambda_i$'s.
{
\renewcommand{\thetheorem}{\ref{lem:V_s}}
\begin{theorem}
If $K$ bounds a disk $D$ in $W$ such that $[D]=(a_1,\cdots, a_n)$, then for $(\lambda_1,\cdots,\lambda_n)$ such that $\lambda_i$'s are odd and $0 \le \sum_{i=1}^n \lambda_i a_i \le k$, we have that
$$\sum_{i=1}^n \lambda_i^2-n\ge 8V_j(K),$$
where $j=\frac{1}{2}\left(k-\sum \lambda_ia_i\right)$.
\end{theorem}
\addtocounter{theorem}{-1}
}
\begin{proof}
Suppose $[D]=(a_1,\cdots,a_n) \neq 0$. Consider the inequality \eqref{family} with $[D]=(a_1,\cdots,a_n)$ and spin$^c$ structure $(\lambda_1,\cdots,\lambda_n)$. If $0 \le \sum_{i=1}^n \lambda_i a_i \le k$, then $\sum_{i=1}^n \lambda_i a_i \equiv k-2j \pmod{2k}$ implies that $2j=k-\sum_{i=1}^n \lambda_i a_i$ in $\mathbb{Z}$, and therefore $(2j-k)^2=\left(\sum_{i=1}^n\lambda_i a_i\right)^2$. Moreover, since $V_s$ is non-increasing on $s$, the inequality $2j\le k$ implies that $V_j(K)\ge V_{k-j}(K)$. Thus, the proposition holds for $[D]\neq 0$.

Now, consider the case where $D$ is null-homologous, i.e. all the $a_i$'s are zero. Attach a $(-1)$-framed $2$-handle along the meridian of the knot, and obtain a slice disk of $K$ with homology class $(0,\cdots,0,1)$ in $W \#\, \CPbone$. It follows from above that $V_0(K)= 0$. Since all $\lambda_i$'s are odd, we have that $\sum_{i=1}^n \lambda_i^2-n\ge 0$ and $j=0$, for any choice of $\lambda_i$'s. The result follows.
\end{proof}

As a corollary, we obtain the adjunction inequality for the knot invariant $\nu_+$ for the case where the surface is a disk. 
\begin{corollary}\label{nu+}
 If a knot $K$ bounds a disk $D$ in $W$ such that the homology class of $D$ is $(a_1,...,a_n)$ in $H^2(W)$, we have that
\begin{equation}
2\nu_+(K) \le \sum_{i=1}^na_i^2 - \sum_{i=1}^n|a_i|.
\end{equation}
\end{corollary}

\begin{proof}
Theorem \ref{lem:V_s} with $\lambda_i= \text{sign}(a_i)$ implies that $V_{j}(K)=0$ for $j=\frac{1}{2}\left(\sum a_i^2-\sum |a_i|\right)$. Recall that $\nu_+=\min\{s\: |\: V_s =0\}$. Thus, we have that $\nu_+(K) \le \frac{1}{2}\left(\sum a_i^2 - \sum|a_i|\right).$
\end{proof}

Since $\tau(K) \le \nu_+(K)$, this strengthens the obstruction coming from the adjunction inequality of $\tau$ in \cite[Theorem 1.1]{OS03}.
\begin{example}\label{v00}
Let $K$ be the knot $T_{3,7}\# \overline{T_{4,5}}$. By \cite[Proposition 2.4]{Bod} $\nu_+(K)=\nu_+(m(K))=1$. From Corollary \ref{nu+}, we deduce that $K$ cannot be $1$-slice, $2$-slice or $3$-slice. Moreover, since $\nu_+(m(K))=1$, $sd_+(m(K))$ is also nonzero. This gives an example where the slicing degree does not vanish for both the knot and its mirror. 
\end{example}

\begin{proposition}\label{odd}
If $K$ bounds a disk $D$ in $W$ such that $[D]=(a_1,\cdots, a_n)$ with $a_i$ odd. then
$$
k-n \ge 8 V_0(K).
$$
\end{proposition}
\begin{proof}
Choose $\lambda_i$ to be $a_i$. Then, apply Theorem \ref{lem:V_s}.
\end{proof}

We compute the invariants $\{V_s\}_{s\ge 0}$ of Floer thin knots and $L$-space knots, whose full knot Floer complexes can be read from their Alexander-Conway polynomials. Recall that a Floer thin knot is one whose knot Floer homology is supported in a single diagonal. We will see that the values $\{V_s\}_{s\ge 0}$ of a Floer thin knot are determined by its $\tau$-invariant. The case for alternating knots was considered in \cite[Example 7.1]{allen}.

\begin{proposition}\label{vs:thin}
Let $K$ be a Floer thin knot. If $\tau(K)$ is positive, then
$$V_s(K) = \max\left\{\left\lfloor\frac{\tau(K)+1-s}{2}\right\rfloor, 0\right\}. $$
Otherwise, $V_s(K)=0$ for all $s\ge 0$.
\end{proposition}

\begin{proof}
Since $K$ is a thin knot, the complex $C=CFK^{\infty}(K)$ is homotopy equivalent to a direct sum of a single stair case and some number of square complexes \cite[Lemma 7]{Ina} (See \cite[Figure 14]{HM}). The homology of the square complexes vanish when the gradings are sufficiently large, so it suffices to consider the staircase for computing the $V_s$'s. 

If $\tau(K)$ is non-negative, then by \cite[Lemma 1]{zhang23d}, the homology of the staircases restricted to $A_s^+$ (resp. $B^+$) is nontrivial if and only if the staircase is fully contained in $A_s^+$ (resp. $B^+$). Therefore, the tower in $H_*(B^+)$ is generated by the staircase with $i$-grading $\tau(K)$. By observing the staircases in \cite[Figure 14]{HM}, we have that the tower of $A_s^+$ is generated by the staircase with $i$-grading $\min\left\{\left\lfloor\frac{\tau(K)+s}{2}\right\rfloor, \tau(K)\right\}$. Hence, $V_s(K) = \max\left\{\left\lfloor\frac{\tau(K)+1-s}{2}\right\rfloor, 0\right\} $.
If $\tau(K)$ is negative, then the staircase within $C$ is of the same shape as a staircase in the mirror of some $L$-space knot. Thus, from the proof of \cite[Proposition 7.3]{HM}, we have that $V_0(K)=0$.
\end{proof}

An $L$-space knot is one which admits an $L$-space surgery. In \cite{OS05}, Ozsv{\'{a}}th and Szab{\'{o}} showed that the Alexander polynomial of an $L$-space knot $K$ is of the form 
$$\Delta_K(t)=(-1)^m + \sum_{i=1}^m (-1)^{m-i}(t^{n_i}+t^{-n_i})$$
for a sequence of positive integers $0<n_1<n_2< \cdots <n_m$. They also proved that the full knot Floer complex of an $L$-space knot $K$ is generated by a staircase, where the lengths of the stairs are determined by the $n_i$'s (See \cite[Figure 12]{HM}). Hence, we can compute the invariants $\{V_s\}_{s\ge 0}$ of $K$.
\begin{proposition}
Let $K$ be an $L$-space knot, such that $\tau(K)$ is positive. Suppose $\Delta_K(t)=(-1)^m + \sum_{i=1}^m (-1)^{m-i}(t^{n_i}+t^{-n_i})$. For $1\le k\le m$, let $l_k$ be $n_k-n_{k-1}$. Let $n(K)$ be the alternating sum
$n_m-n_{m-1}+ \cdots + (-1)^{m-2}n_2+(-1)^{m-1}n_1$. Moreover, let $l_0=0$ and let $l_{m+1}= + \infty$. 
If $m$ is odd, then
$$V_s(K) = n(K)-\max_{1 \leq i \leq N}\left\{\min\left\{\sum_{k=0}^{i-1} l_{2k+1}, s- \sum_{k=0}^{i-1} l_{2k}\right\}\right\},$$
for $s\in\left[\sum_{k=0}^{N-1} l_{2k},\sum_{k=0}^{N} l_{2k}\right)$ with $1 \le N \le \lceil m/2 \rceil$.
If $m$ is even, then
$$V_s(K) = n(K)-\max_{0 \leq i \leq N}\left\{\min\left\{\sum_{k=0}^{i} l_{2k}, s- \sum_{k=0}^{i-1} l_{2k+1}\right\}\right\},$$
for $s\in\left[\sum_{k=0}^{N-1} l_{2k+1},\sum_{k=0}^{N} l_{2k+1}\right)$ with $0 \le N \le m/2$.
\end{proposition}
\begin{proof}
The staircases shift by $(1,1)$ when multiplied by $U^{-1}$. Since the homology of the staircase vanishes unless it is fully contained in $A_s^+$ (resp. $B^+$) by \cite[Lemma 1]{zhang23d}, we want to find the least number of multiplications by $U^{-1}$ such that the whole staircase is included in $A_s^+$(resp. $B^+$). It is easy to see that the least number of shifting for $B^+$ is $n(K)$, which is the width of the staircase. 

The computation for $A_s^+$ depends on the parity of $m$. If $m$ is odd, for $s\in \left[\sum_{k=0}^{N-1} l_{2k},\sum_{k=0}^{N} l_{2k}\right)$ with $0 \le N \le \lceil m/2\rceil$, the generators $x_1^1, \cdots, x_{2N-1}^1$ are under the line $\{j=s\}$. The number of moves to shift $x_{2i-1}$ into the complex $A_s^+$ is $\min\{ \sum_{k=0}^{i-1} l_{2k+1}, s-\sum_{k=0}^{i-1}l_{2k}\}$. Thus, the least number of shifting needed is $$\max_{1\le i\le N}\left\{ \min\left\{ \sum_{k=0}^{i-1} l_{2k+1}, s-\sum_{k=0}^{i-1}l_{2k}\right\} \right\}.$$

Similarly, if $m$ is even, then for $s\in \left[\sum_{k=0}^{N-1} l_{2k+1},\sum_{k=0}^{N} l_{2k+1}\right)$ with $0 \le N \le m/2$, the generators $x_2^1, \cdots, x_{2N}^1$ are under the line $\{j=s\}$. The number of moves to shift $x_{2i}$ into the complex $A_s^+$ is $\min\{ \sum_{k=1}^i l_{2k}, s-\sum_{k=0}^{i-1}l_{2k+1}\}$. Thus, the least number of shifting needed is $$\max_{0\le i\le N}\left\{ \min\left\{ \sum_{k=1}^i l_{2k}, s-\sum_{k=0}^{i-1}l_{2k+1}\right\} \right\}.$$ 

To obtain $\{V_s\}_{s \ge 0}$, take the difference between the moves required for $A_s^+$ and $B^+$.
\end{proof}

\begin{proposition}
Let $K$ be the mirror of an $L$-space knot. Then $V_s(K)=0$ for all $s$.
\end{proposition}
\begin{proof}
By \cite[Proposition 7.3]{HM}, $V_0(K)=0$. Since $\{V_s\}_{s\ge 0}$ is non-increasing, we have that $V_s(K)=0$ for all $s$.
\end{proof}

{
\renewcommand{\thetheorem}{\ref{thm:t22m+1}}
\begin{theorem}
If $K$ is a Floer thin knot, then $sd_+(K)$ is greater or equal to $4\tau(K)$. In particular, the equality is achieved for $T_{2, 2m+1}$ with $sd_+(T_{2, 2m+1}) = 4m$.
\end{theorem}
\addtocounter{theorem}{-1}
}
\begin{proof} Let $K$ be a thin knot with $\tau(K)\ge 0$. Suppose $K$ bounds a disk $D$ in {\small $\CPbar{n}\punc$}. Upon permuting the coordinates and changing signs, we can assume that the homology class of $D$ is $[a_1,\cdots, a_m, a_{m+1} \cdots, a_{n}],$ such that $a_i$'s are non-negative, $\{a_1,\cdots,a_m\}$ are even and $\{a_{m+1} \cdots, a_{n}\}$ are odd. 

Apply Theorem \ref{lem:V_s} to the case where $(\lambda_1,\cdots, \lambda_n) = (a_1 -1, \cdots, a_m -1, a_{m+1},\cdots,a_{n})$, and compute that
$$\sum_{i=1}^m (a_i-1)^2 + \sum_{i=m+1}^n a_i^2 - n \ge 8V_j(K),$$
where $j = \frac{\sum_{i=1}^m a_i}{2}$.
By Proposition \ref{vs:thin}, we have that $V_j(K)=\max\left\{\left\lfloor\frac{\tau(K)+1-j}{2}\right\rfloor, 0\right\}$. Hence, if such disk $D$ exists, then we have that
$$k - 2\sum_{i=1}^m a_i - (n-m) \ge 8\max\left\{\left\lfloor\frac{\tau(K)+1-\frac{\sum_{i=1}^m a_i}{2}}{2}\right\rfloor, 0\right\}.$$
However, if $k < 4\tau(K)$, then the LHS would be strictly less than $$4\tau - 2\sum_{i=1}^m a_i = 8\left(\frac{\tau(K)-\frac{\sum_{i=1}^m a_i}{2}}{2}\right).$$ Contradiction. Thus, the slicing degree of a thin knot $K$ is bounded below by $4\tau(K)$.

Since the torus knot $T_{2, 2m+1}$ is alternating and its $\tau$-invariant is $m$, we have that $$sd_+(T_{2, 2m+1})\ge 4m.$$ By Lemma \ref{lem:sn}, since the unknotting number of $T_{2, 2m+1}$ is $m$, we conclude that $sd_+(T_{2, 2m+1})=4m$.
\end{proof}
\subsection{Singular instanton invariants $\Gamma_K(s)$}
Daemi and Scaduto in \cite{daemi19} constructed the concordance knot invariant $\Gamma_K(s), s\in \mathbb{Z}$ from the Chern-Simons filtration on the equivariant singular instanton complex of a knot $K$. By \cite[Proposition 4.33]{daemi20}, we have that the knot invariant $\Gamma_K$ obstructs $K$ from bounding certain immersed surfaces in some negative-definite $4$-manifold. For instance, \cite[Corollary 4.43]{daemi20} proves that the knot $7_4$ does not bound a disk of homology class $(2)$ in $\CPbone\punc$. 

Let $W$ be a negative-definite $4$-manifold with $b_1(W) =0$ and $b_2(W) = n$. Suppose $K$ bounds a disk $D$ in $W^{\circ} = W \punc$, such that the homology class of $D$ is $(a_1,\cdots,a_n)$ in $H^2(W;\mathbb{Z})$. Then, after removing a small neighborhood of a point on the disk, we obtain a negative-definite cobordism pair $(W, D)$ from $(S^3, U)$ to $(S^3, K)$. Let $c=(c_1,\cdots, c_n)$ be a cohomology class in $H^2(W; \mathbb{Z})$. The minimal topological energy $\kappa_{min}$ among all reducible instantons is defined as
$$\kappa_{min}(W, D, c)=\min\left\{\sum_{i=1}^n \left.\left (z_i + \frac{a_i}{4} - \frac{c_i}{2}\right )^2\right\vert (z_1, \cdots, z_n) \in \mathbb{Z}^n\right\}.$$ Define $\Phi_{min}(W, D, c)$ to be the set of minimal reducibles (i.e. reducibles that minimize the topological energy):
$$\Phi_{min}(W, D, c)=\left\{(z_1, \cdots, z_n) \in \mathbb{Z}^n\left\vert \: \kappa_{min}(W, D, c) = \sum_{i=1}^n \left (z_i + \frac{a_i}{4} - \frac{c_i}{2}\right )^2 \right.\right\}.$$  
To keep track of the monopole number of  minimal reducibles, we define a signed count $\eta(W, D, c)$ of minimal reducibles to be $$\eta(W, D, c)=\sum_{z\in \Phi_{min}(W, D, c)} (-1)^{\mu(z)}T^{\nu(z)}\in \mathbb{Z}[T^{\pm}],$$
where $\mu(z) = - \sum z_i^2$ is the self-intersection number of $z$, and $\nu(z)=\sum a_i(c_i-2z_i)$ is the monopole number of $z$. Conventionally, we omit $c$ when $c=0$.
\begin{proposition}[Proposition 4.33 in \cite{daemi20}]\label{prop: daemi20}
Suppose $K$ bounds a disk $D$ in $W$ such that $[D]=(a_1, \cdots, a_n)\in H^2(W;\mathbb{Z})$. Let $c=(c_1,\cdots, c_n)\in H^2(W; \mathbb{Z})$. Let $$i=4\kappa_{min}(W, D, c) - \frac{1}{4} \sum_{i=1}^n a_i^2 - \frac{1}{2}\sigma(K).$$
If $\eta(W, D, c)\neq 0 \in \mathbb{Z}[T^{\pm}]$ and $i\ge 0$, then
$$\Gamma_K (i) \le 2\kappa_{min}(W, D, c).$$
\end{proposition}

\begin{proposition}\label{applydaemi}
Let $K$ be a knot such that $\sigma(K) \le 0$. Suppose $K$ bounds a disk $D$  such that $[D] = (\,\underbrace{2,\cdots, 2}_{p}\: , \underbrace{1,\cdots, {1}}_{q}\:, 0 \cdots, 0)$ in $\CPbar{n}$ with $p,q\ge 0$ and $n \ge p+q$. Then, $$\Gamma_K\left(-\frac{1}{2}\sigma(K) \right) \le \frac{p}{2} + \frac{q}{8}=\frac{k}{8}.$$ 
\end{proposition}
\begin{proof}
Denote $\CPbar{n}$ by $X$. Suppose $K$ bounds such a disk $D$. Let $c=0\in H^2(X;\mathbb{Z})$. Then, the minimal topological energy is
$$\kappa_{min}(X, D)=\min\left\{\sum_{i=1}^p\left (z_i + \frac{1}{2}\right )^2 + \sum_{i=p+1}^{p+q}\left(z_i + \frac{1}{4}\right)^2 + \sum_{i=p+q+1}^n z_i^2 \,\right\} = \frac{p}{4} + \frac{q}{16},$$
and the minimum is achieved at $$\Phi_{min}(X, D)=\{(z_1,\cdots,z_p,0,\cdots, 0)|\: z_i\in\{0,-1\}\}.$$
Moreover, we compute that $\eta(X,D) =\left (1-T^4\right)^p$ is nonzero in $\mathbb{Z}[T^{\pm}]$ and $i = -\frac{1}{2}\sigma(K)$. We deduce from Proposition \ref{prop: daemi20} that
$$\Gamma_K\left(-\frac{1}{2}\sigma(K) \right) \le \frac{p}{2} + \frac{q}{8}=\frac{k}{8}.$$
\end{proof}
\begin{example}\label{ex: 7495}
Consider $7_4$ and $9_5$, both of which are double twists knots with signature $\sigma = -2$. The knot $7_4$ is the double twist knot $D_{2,2}$ and the knot $9_5$ is the double twist knot $D_{2,3}$. \cite[Corollary 3.24]{daemi20} computes that $$\Gamma_{D_{m,n}}(1) = \frac{(2m-1)(2n-1)}{4mn-1}.$$ 

In particular, we have that $\Gamma_{7_4}(1) =\frac{3}{5}$ and $\Gamma_{9_5}(1)=\frac{15}{23}$. Since $\Gamma_{7_4}(1)$ is greater than $\frac{1}{2}$, it follows from Proposition \ref{applydaemi} that $7_4$ cannot bound a disk of class $(1,0,\cdots, 0)$, $(1,1,0,\cdots, 0)$, $(1,1,1,0,\cdots, 0)$, $(1,1,1,1,0,\cdots, 0)$ or $(2,0,0\cdots, 0)$. Thus, we get that $sd_+(7_4)\ge 5$. 

Similarly, since $\Gamma_{9_5}(1)$ is greater than $\frac{1}{2}$, we have that $sd_+(9_5)\ge 5$. Moreover, since $\Gamma_{9_5}(1)=\frac{15}{23} > \frac{5}{8}$, the knot $9_5$ cannot bound a disk of class $(1,1,1,1,1,0,\cdots, 0)$ or $(2,1, 0,\cdots, 0)$. Thus, we conclude that $sd_+(9_5)\ge 6$.
\end{example}

\begin{remark}
Note that $s(7_4)=s(9_5)=2$. By Proposition \ref{stau}, we only get that $sd_+(7_4) \ge 4$ and $sd_+(9_5) \ge 4$. Moreover, since both $7_4$ and $9_5$ are alternating with $\tau = 1$, Proposition \ref{vs:thin} implies that they have the same $\{V_s\}_{s \ge 0}$ as the right-handed trefoil. Thus, the lower bound we get from the knot Floer homology is also $sd_+(7_4) \ge 4$ and $sd_+(9_5) \ge 4$. Hence, Example \ref{ex: 7495} shows that the invariants $\Gamma_K(s)$ can provide stronger lower bounds than the $s$-invariant or the invariants $\{V_s\}_{s\ge 0}$. 
\end{remark}

\begin{example}\label{910}
The knot $9_{10}$ is the $2$-bridge knot $S(33,10)$ with $\sigma = -4$. Suppose $K$ bounds a disk in homology class $(2,2,0, \cdots, 0)$. By Proposition \ref{applydaemi}, we should have that $$\Gamma_K(2) \le 1.$$ However, one can compute that $\Gamma_{9_{10}}(2)=\frac{36}{33}>1$. Combine with the fact that $s(9_{10})=4$, we conclude that $sd_+(9_{10})\ge 9$.
\end{example}

\section{k-friends}\label{friends}
Two knots $K$ and $K'$ are called $k$-friends if their $k$-surgeries are homeomorphic. The obstructing methods in Section \ref{lowb} are not sensitive to distinguishing consecutive sliceness. For instance, consider a knot $K$ with slicing degree $sd_+(K)=1,2$ or $3$. Then, $s(K)\le 0$ by Theorem \ref{prop:sinv}, and $V_0\le 0$ by Theorem \ref{lem:V_s}. Although $\Gamma_K(s)$ has the potential to distinguish between $1,2$ and $3$-sliceness, $\Gamma_K(s)$ is hard to compute in general.

However, as seen in the proof of \cite[Theorem 1.5]{Qin}, we can obstruct $k$-sliceness by finding a $k$-friend for $K$, when the knot invariants of $K$ itself do not obstruct $k$-sliceness directly. 

\begin{definition}
Let $L=\{(R,r),(B,0),(G,0)\}$ be a $k$-special RBG link such that $R$ is $r$-slice in some $\#^{n}{\mathbb{CP}^2}$,and denote the knot pair associated to $L$ by $(K,K')$. Then, $L$ is called a {\em $k$-special friendship} between $K$ and $K'$. Furthermore, we say that $K$ and $K'$ are {\em $k$-special friends}.
\end{definition}
Under this terminology, we rephrase \cite[Theorem 1.4(a)]{Qin} as follows.
\begin{theorem}\label{thm: kfriend}
Suppose $K$ has a $k$-special friend $K'$. If $s(K') > k-\sqrt{k}$, then $sd_+(K)> k$.
\end{theorem}
\begin{proposition}
Define $K_B(m)$ by the knot diagram to the left of Figure \ref{KBmKG}. Then,  
$$sd_+(K_B(m))=m+1,$$
for $m\in\{-1,0,1,2\}$.
\end{proposition}
\begin{proof}
Fix $m\in \{0,1,2\}$. Notice that $K_B(m)$ is the knot $K_B(-2,1,m;*)$ in \cite[Figure 14]{Qin}. Let $L(m)$ be the $m$-special RBG link $L(-2,1,m; m-2)$ in \cite[Figure 13]{Qin}. From \cite[Figure 14]{Qin}, we can see that $K_G(-2,1,m; m-2)$ has a knot diagram as in the right frame in Figure \ref{KBmKG}, which does not depend on $m$. Denote the knots $K_G(-2,1,m; m-2)$ by $K_G$. Since the $R$-component of $L(m)$ is the unknot and $m\in \{0,1,2\}$, we have that $R$  is $(m-2)$-slice in $\#^{(2-m)} {\mathbb{CP}^2}$. Thus, $K_G$ is an $m$-special friend of $K_B(m)$. KnotJob \cite{KnotJob} computes that $s(K_G) = 2$, which is greater than $m - \sqrt{m}$. By Theorem \ref{thm: kfriend}, we deduce that $$sd_+(K_B(m))> m.$$ 

On the other hand, the knot $K_B(-1)$ is recognized by SnapPy \cite{SnapPy} as $K11n139$, which is also known as the pretzel knot $P(5, 3, -3)$. Theorem 1.1 in \cite{slicep} implies that the knot $P(5, 3, -3)$ is ribbon. Thus, we have that $sd_+(K_B(-1))=0$. Notice that $K_B(m)$ can be obtained from $K_B(-1)$ by adding $(m+1)$-many positive full-twists at the twist box labelled by $c$ in \cite[Figure 14]{Qin}. Since the algebraic intersection number of twisting box $c$ is one, we have that $K_B(m)$ bounds a disk with homology class $(1,\cdots, 1)$ in $\#^{(m+1)} \overline{\mathbb{CP}^2}$, which is a $(m+1)$-slice disk for $K_B(m)$.

Thus, we conclude that $$sd_+(K_B(m))=m+1,$$
for $m\in\{-1,0,1,2\}$.
\end{proof}
\begin{figure}[h]
{
   \fontsize{9pt}{11pt}\selectfont
   \def\svgwidth{4in}
   \begin{center}
   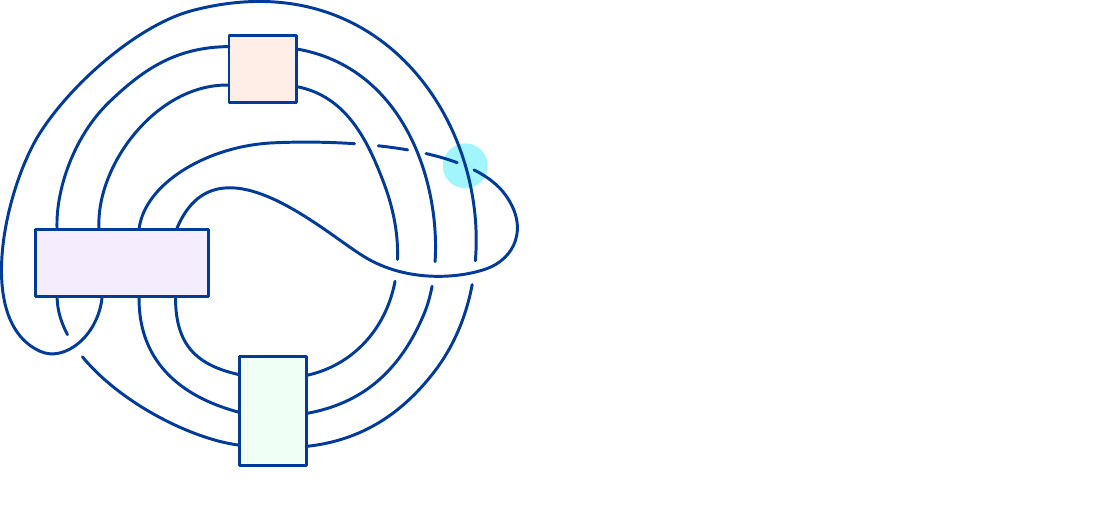
   \end{center}
   \caption{}
   \label{KBmKG}
}
\end{figure}
\begin{remark}\label{rm:942}
The knot $K_B(0)$ is the knot $9_{42}$ by SnapPy \cite{SnapPy}. We will see an alternative proof of $sd_+(9_{42}) = 1$ in Proposition \ref{prop:942} below.
\end{remark}
Notice that the obstructions from Heegaard Floer invariants and singular instanton Floer invariants apply to any simply-connected negative-definite 4-manifold. Therefore, we can combine those invariants with $k$-friendships to obstruct $k$-sliceness of a knot $K$.
\begin{lemma} 
Let $K'$ be a $k$-friend of $K$. If $K'$ is not $k$-slice in any simply-connected negative-definite 4-manifold, then $sd_+(K)>k$.
\end{lemma}
\begin{proof}
Suppose the slicing degree of $K$ is less or equal to $k$. Then, $K$ is $k$-slice in some $\#^n\overline{\mathbb{CP}^2}$. By \cite[Proposition 5.3]{Qin} and \cite[Proposition 5.4]{Qin}, since $K'$ is a $k$-friend of $K$, we have that $K'$ is $k$-slice in some simply-connected negative-definite manifold. Contradiction.
\end{proof}

\section{Small knots}\label{small}
In this section, we apply the methods mentioned so far and compute the slicing degrees of most knots with crossing number up to $8$ in Table \ref{08} and the slicing degrees of most knots with crossing number $9$ in Table \ref{09}. 

Note that Rolfsen's table classifies knots up to mirror. Here, our convention is that for knots with non-zero signatures, the knots in Tables \ref{08} and \ref{09} are chosen to be ones with negative signature. Among all the knots with $\sigma = 0$, all except for the knot $8_{18}$, are shown to be BPH-slice knots in \cite[Section 2]{MP21}. Since $sd_+(K) = sd_+(m(K)) = 0$ for BPH-slice knots, it does not matter which one we picked.  The knot $8_{18}$ is amphicheiral, so it does not matter which one we picked either.

By \cite[Example 2.7]{MP21}, there are $24$ BPH-slice prime knots with crossing number at most $9$ (including the unknot). The following example gives the answer for $40$ of the knots in Table \ref{08} and Table \ref{09}.
\begin{example}
If $s(K)=2$ and $c_s(K)=1$, then $sd_+(K) = 4$. If $s(K)=4$ and $c_s(K)=2$, then $sd_+(K) = 8$. 
\end{example}
For knots with larger $s$-invariant, we use the obstruction given by Theorem \ref{thm:t22m+1}.
\begin{example}
If $K$ is an alternating knot with $\tau(K)=3$ and $c_s(K)=3$, then $sd_+(K) = 12$. For example, $K$ can be $7_1$, $9_3$, $9_6$, $9_9$ or $9_{16}$.
\end{example}

\begin{example}
The knot $9_1$ is the torus knot $T_{2,9}$. By Theorem \ref{thm:t22m+1} with $m=4$, we have that $sd_+(9_{1}) = 16$. The knot $8_{19}$ is the torus knot $T_{3,4}$. By Proposition \ref{prop:tmm+1}, we have that $sd_+(8_{19}) = 9$.
\end{example}
Although it suffices to look at the clasp number for most of small knots, we still need to keep track of the sign of the crossing changes for some. For example, since the knot $8_{18}$ can be unknotted by changing one positive crossing and one negative crossing, its slicing degree is at most $4$. Similarly, for knots with non-vanishing $s$-invariants (or $\tau$-invariant), we have the following.
\begin{example}\label{+casp}
If $s(K)=2$ (or $\tau(K)=1$) and $K$ becomes a $BPH$-slice knot by changing one positive crossing, then $sd_+(K) = 4$ and $sd_+(m(K))=0$. For example, $K$ can be $8_{16}$, $9_{15}$, $9_{17}$, ${9_{31}}$, $9_{40}$. 
\end{example}

The knot $9_{42}$ as mentioned in Remark \ref{rm:942} has slicing degree one. The following proposition provides an alternative proof of this fact.
\begin{proposition}\label{prop:942}
The slicing degree of the knot $9_{42}$ is one.
\end{proposition}
\begin{proof}
Since the knot $9_{42}$ can be obtained by adding a positive twist to the Mazur pattern, the knot $9_{42}$ is $1$-slice in {\small $\CPbone$} and thus $sd_+(9_{42})\le 1$. Since the signature of $9_{42}$ is $-2$, the knot $9_{42}$ is not $H$-slice in any $\#^n\overline{\mathbb{CP}^2}$ and therefore $sd_+(9_{42})\neq 0$. 
\end{proof}
\begin{remark}
In \cite{SSS}, Sarkar–Scaduto–Stoffregen constructed a stable homotopy type for odd Khovanov homology, and they defined the concordance invariants  $r_{\pm}^{\alpha}$ and $s_{\pm}^{\alpha}$ using the mod-2 Steenrod algebra for both even and odd Khovanov homotopy type, generalizing \cite[Definition 1.2]{LSsinv} in the case of $\mathbb{Z}/2$-coefficients.  

One can compute using KnotJob \cite{KnotJob} that the invariants $r^{\Sq^1_{\odd}}_+$, $s^{\Sq^1_{\odd}}_+$ and $s^{\Sq^2_{\even}}_+$ evaluate to $2$ for the knot $9_{42}$. This shows that these invariants do not obey the same adjunction inequality as the $s$-invariants in Theorem \ref{prop:sinv}. 
\end{remark}
We are left with $8$ knots (namely $7_4$, $8_{18}$, $9_5$, $9_{10}$, $9_{13}$, $9_{35}$, $9_{38}$ and $9_{49}$), whose slicing degrees are still unknown. 
For the $2$-bridge knots $7_4$, $9_5$ and $9_{10}$, we can use the invariants $\Gamma_K(s)$ to raise the lower bounds as in Example \ref{ex: 7495} and Example \ref{910}.

Finally, one can check that the mirrors of the knots in Table \ref{08}, \ref{09} (except for $8_{18}$) can be unknotted by changing negative crossings to positive crossings, and therefore their slicing degrees are all zero.

\begin{table}
\begin{center}
\begin{tabular}{||c | c || c | c || c | c || c | c ||} 
 \hline 
 $K$ & $sd_+(K)$ & $K$ & $sd_+(K)$ & $K$ & $sd_+(K)$ & $K$ & $sd_+(K)$ \\ [0.5ex] 
 \hline\hline
 $0_1$ & $0$ & $7_2$ & $4$  & $8_4$ & $4$ & $8_{13}$ & $0$ \\ 
 \hline
 $3_1$ & $4$ & $7_3$ & $8$  & $8_5$ & $8$ & $8_{14}$ & $4$\\
 \hline
 $4_1$ & $0$ &$7_4$ & $[5,8]$ & $8_6$ & $4$ & $8_{15}$ & $8$ \\
 \hline
 $5_1$ & $8$ & $7_5$ & $8$ &  $8_7$ & $4$ & $8_{16}$ & $4$ \\
  \hline
 $5_2$ & $4$ & $7_6$ & $4$ &  $8_8$ & $0$ & $8_{17}$ & $0$ \\
  \hline
 $6_1$ & $0$ & $7_7$ & $0$ &  $8_9$ & $0$ & $8_{18}$ & $[0,4]$ \\
  \hline
 $6_2$ & $4$ & $8_1$ & $0$ &  $8_{10}$ & $4$ & $8_{19}$ & $9$ \\
  \hline
 $6_3$ & $0$ & $8_2$ & $8$ &  $8_{11}$ & $4$ & $8_{20}$ & $0$ \\
  \hline
 $7_1$ & $12$ & $8_3$ & $0$ &  $8_{12}$ & $0$ & $8_{21}$ & $4$ \\
 \hline
\end{tabular}
\end{center}
\caption{Slicing degrees of knots up to $8$ crossings}
\label{08}
\end{table}

\begin{table}
\begin{center}
\begin{tabular}{||c | c || c | c || c | c || c | c || c | c || c | c ||  c | c ||} 
 \hline 
 $K$ & $sd_+$ & $K$ & $sd_+$ & $K$ & $sd_+$ & $K$ & $sd_+$ & $K$ & $sd_+$ & $K$ & $sd_+$ & $K$ & $sd_+$\\ [0.5ex] 
 \hline\hline
 $9_1$ & $16$ & $9_8$ & $4$  & $9_{15}$ & $4$ & $9_{22}$ & $4$ & $9_{29}$ & $4$ & $9_{36}$ & $8$ & $9_{43}$ & $8$\\ 
 \hline
 $9_2$ & $4$ & $9_9$ & $12$  & $9_{16}$ & $12$ & $9_{23}$ & $8$ & $9_{30}$ & $0$ & $9_{37}$ & $0$& $9_{44}$ & $0$\\
 \hline
 $9_3$ & $12$ &$9_{10}$ & $[9,12]$ & $9_{17}$ & $4$ & $9_{24}$ & $0$ & $9_{31}$ & $4$ & $9_{38}$ & $[8, 12]$& $9_{45}$ & $4$\\
 \hline
 $9_4$ & $8$ & $9_{11}$ & $8$ &  $9_{18}$ & $8$ & $9_{25}$ & $4$ & $9_{32}$ & $4$ & $9_{39}$ & $4$& $9_{46}$ & $0$\\
  \hline
 $9_5$ & $[6,8]$ & $9_{12}$ & $4$ &  $9_{19}$ & $0$ & $9_{26}$ & $4$ & $9_{33}$ & $0$ & $9_{40}$ & $4$& $9_{47}$ & $4$\\
  \hline
 $9_6$ & $12$ & $9_{13}$ & $[8,12]$ &  $9_{20}$ & $8$ & $9_{27}$ & $0$ & $9_{34}$ & $0$ & $9_{41}$ & $0$& $9_{48}$ & $4$\\
  \hline
 $9_7$ & $8$ & $9_{14}$ & $0$ &  $9_{21}$ & $4$ & $9_{28}$ & $4$ & $9_{35}$ & $[4,8]$ & $9_{42}$ & $1$& $9_{49}$ & $[8,12]$\\
  \hline
\end{tabular}
\end{center}
\caption{Slicing degrees of knots with $9$ crossings}
\label{09}
\end{table}
\bibliography{slicing_degree}
\bibliographystyle{custom}

\end{document}